\documentclass[12pt]{article}

\usepackage{pgf,tikz}
\usepackage{mathrsfs}
\usetikzlibrary{arrows}

\usepackage{graphicx} %Loading the package
\graphicspath{{figures/}} %Setting the graphicspat

\usepackage{fullpage, url,amsmath,amsfonts,amssymb,mathtools,mathrsfs,graphicx,  algorithm, float, sansmath,epstopdf,color,caption,enumitem,tabularx}
\usepackage[final]{pdfpages}
\usepackage{amsthm}
\usetikzlibrary{automata,topaths}
\usetikzlibrary{decorations.pathreplacing,shapes.misc}
\usepackage{fancyhdr}

\def\beq{\begin{equation}}
\def\eeq{\end{equation}}
\def\baq{\begin{eqnarray}}
\def\eaq{\end{eqnarray}}
\def\baqn{\begin{eqnarray*}}
\def\eaqn{\end{eqnarray*}}

\newcommand{\ball}{\mathbb{B}}

\usepackage{multirow}
\usetikzlibrary{calc,arrows}
\theoremstyle{plain}
\newtheorem{definition}{Definition}
\newtheorem{remark}{Remark}

\newtheorem{theorem}{Theorem}
\newtheorem{lemma}[theorem]{Lemma}

\newtheorem{proposition}{Proposition}

\usepackage{blindtext}

\usepackage[colorlinks,linkcolor=blue,citecolor=red]{hyperref}

\usepackage{xcolor, framed}

\newcommand{\R}{{\mathbb R}}

\newcommand{\interior}{{\rm int}\kern 0.06em}

\def\<{\langle}
\def\>{\rangle}

\newcommand\dom{{\rm dom}}%

\newcommand{\gra }{\operatorname{gra}}
\newcommand{\ran }{\operatorname{ran}}
\newcommand{\zer}{\operatorname{zer}}

\usepackage{ulem}

\usepackage{subcaption}

\usepackage{pict2e}

\if
{

\usepackage[pageref]{backref}
\renewcommand*{\backrefalt}[4]{%
\ifcase #1 %
(Not cited)%
\or
(Cited on p.~#2)%
\else
(Cited on pp.~#2)%
\fi
}

}
\fi

\begin{document}
\title{Inexact Proximal Point and Tseng Algorithms with Nonsummable Errors to Solve Monotone Inclusions}
\author{Ba Khiet Le \thanks{Analytical and Algebraic Methods in Optimization Research Group, Faculty of Mathematics and Statistics, Ton Duc Thang University, Ho Chi Minh City, Vietnam.
 E-mail: \texttt{lebakhiet@tdtu.edu.vn}. Research of this author was  supported by National Foundation for Science and
Technology Development (NAFOSTED) of Vietnam under Grant Number
101.01-2025.61}  \qquad Boris S. Mordukhovich\thanks{Department of Mathematics and Center for Artificial Intelligence and Data Science, Wayne State University, Detroit, MI 48202, USA.
 E-mail: \texttt{aa1086@wayne.edu}. Research of this author was partly supported by the US National Science Foundation under grant DMS-2204519 and by the Australian Research Council under Discovery Project DP-190100555} \qquad Michel  A. Th\' era \thanks{l XLIM UMR-CNRS 7252 and Universit\'{e} de Limoges,   123 Avenue Albert Thomas,
87060 Limoges CEDEX, France.  E-mail: \texttt{michel.thera@unilim.fr} ORCID 0000-0001-9022-6406}}
%\date{}
\maketitle

%\begin{center}
%{\bf Dedicated to Simeon Reich in honor of his 80th birthday}
%\end{center}\vspace*{0.1in}

\begin{abstract}
In this paper, we establish, for the first time in the literature, the convergence of the practical versions of the Inexact Proximal Point Algorithm (IPPA) and the Inexact Tseng Algorithm (ITA) for computing approximate solutions to monotone inclusions in Hilbert spaces under the the presence of {\it nonsummable errors}. Our approach relies on Tikhonov regularization, the contraction property of the associated monotone operators, and the recently developed R-continuity theory. The proposed techniques and results can be extended to analyze various important inexact algorithms in optimization-related problems  with nonsummable errors.
\end{abstract}



\section{Introduction and Motivations}

This paper concerns the study of efficient algorithms {for solving} the {\it monotone inclusion} 
\begin{equation}\label{main}
0\in Ax,
\end{equation}
where $A:\mathcal{H}\rightrightarrows \mathcal{H}$ is a set-valued {\it maximally monotone} operator defined on a {\it Hilbert space} $\mathcal{H}$. Various algorithms for solving \eqref{main} and related optimization problems have been largely developed over the years; see, e.g.,  \cite{Bauschke,Nesterov,Reich,Rockafellar,tseng} and the references therein. Recent years have witnessed strong attention to {\it inexact} versions of such algorithms aiming to find approximate solutions under errors in computations. The reader can find various inexact algorithms of this type in \cite{Buscaglia,Butnariu,
dev-nest,KLMT,KMT24,KMT26,KMT,L1,LMT,LT,Rz1,Rz2} among other publications.

One of the most classical methods to solve the inclusion problem \eqref{main} is the {\it Proximal Point Algorithm} (\textbf{PPA}) proposed by Rockafellar \cite{Rockafellar} in the form
\begin{equation*}
x_{k+1}=J_{\gamma A}x_k,\qquad x_0\in \mathcal{H},
\end{equation*}
where $J_{\gamma A}:=(I+\gamma A)^{-1}$ denotes the {\it resolvent} of $\gamma A$. It is well known that $J_{\gamma A}$ is {\it firmly nonexpansive}; see, e.g., \cite{Bauschke,Bruck,Goebel,Nesterov,Rockafellar}.

In many practical situations, the resolvent operator cannot be computed exactly and must therefore be approximated numerically. This naturally leads to the {\it inexact iteration}
\begin{equation}\label{pperr}
x_{k+1}=J_{\gamma A}x_k+y_{k+1},\qquad x_0\in \mathcal{H},
\end{equation}
where $(y_k)$ represents an error sequence satisfying the estimate
$\|y_k\|\leq \delta$ with some prescribed tolerance $\delta>0$. This requires to seek constructive conditions ensuring the {\it convergence} of the sequence $(x_k)$ generated by \eqref{pperr}.

Most existing convergence results for inexact proximal methods assume that the error sequence is {\it summable}, i.e., they require that
\begin{equation*}
\sum_{k=1}^{\infty}\|y_k\|<\infty,
\end{equation*}
which implies, in particular, that $y_k\to 0$ as $k\to\infty$. While being mathematically convenient, this assumption is often difficult to guarantee in practical computations. By contrast, the {\it bounded-error condition}
\begin{equation}\label{error}
\|y_k\|\leq \delta\;\mbox{ for all }\;k=1,2,\ldots.
\end{equation}
is significantly more realistic from the computational viewpoint. Other commonly used relative error criteria employed, e.g., in  the aforementioned publications, may also be difficult to verify at each iteration.

Having the above {discussion} in mind, the {\it first goal} of this work is to investigate the {\it well-posedness} of \textbf{PPA} under the {{\it sole  assumption of boundedness}} on the errors \eqref{error}. 
Our approach is based on the {\it Tikhonov regularization} of \eqref{main}. More precisely, we consider the {\it Inexact Proximal Point Algorithm} (\textbf{IPPA}) defined by
\begin{equation*}\label{pperrr}
x_{k+1}
=J_{\gamma \widetilde{A}}x_k+y_{k+1}
=(I+\gamma \widetilde{A})^{-1}x_k+y_{k+1},
\end{equation*}
where $\widetilde{A}$ is an approximation of the operator $A$ given in the form
\[
\widetilde{A}:=A+\epsilon I,
\]
where $I$ is the identity operator, where $\epsilon,\delta>0$ are fixed parameters, and where the error sequence $(y_k)$ satisfies the estimate in \eqref{error}.

Under the assumption that $A^{-1}$ is {\it R-continuous} at the origin (see below), we establish the convergence of $(x_k)$ to an approximate solution of \eqref{main}. This result is particularly relevant in applications, where exact solutions are generally unattainable in finite precision computations, whereas approximate solutions are sufficient for most practical purposes. Our analysis further reveals that the well-posedness and stability of the \textbf{PPA} are essentially induced by the {\it strongly monotone regularization term} rather than by the classical summability assumption on the errors in the previous publications. {To the best of our knowledge, the first work allowing nonsummable but uniformly bounded errors under assumption \eqref{error} is due to  Zaslavski \cite{zas10}; see also \cite{zas12,zas20}. However, obtaining an approximate solution may require a very large number of iterations, see Remark \ref{rm1}-(iv).}

The {\it second goal} of this paper is the extension of the proposed approach to solving the {\it composite monotone inclusion problem} given in the form
\begin{equation*}\label{sumproblem}
0\in Ax+Bx,
\end{equation*}
where $A:\mathcal{H}\rightrightarrows \mathcal{H}$ is a set-valued  maximally monotone operator, and where $B:\mathcal{H}\to\mathcal{H}$ is a single-valued monotone operator. In this framework, we particularly investigate the corresponding {\it Inexact Tseng Algorithm} (\textbf{ITA}) applied to the Tikhonov regularized problem. The key idea consists in rewriting the iteration in the form
\begin{equation}\label{tseng}
x_{k+1}=Tx_k+t_{k+1},\qquad k=0,1,\ldots,
\end{equation}
where $T:\mathcal{H}\to\mathcal{H}$ is a well-defined single-valued operator satisfying a suitable {\it contractive property}, and where the estimate
\[
\|t_{k+1}\|\leq \delta
\]
holds with some given tolerance $\delta>0$.

The rest of the paper is organized as follows. In Section~\ref{s2}, we recall several preliminary notions and results from nonlinear analysis, mainly related to monotone operators and R-continuous mappings. Sections~\ref{s3} and \ref{s4} are devoted to the convergence analysis of \textbf{IPPA} and \textbf{ITA}, respectively. Concluding remarks and some topics of  {future research} are presented in the final Section~\ref{s5}.

\section{Preliminaries from Nonlinear Analysis}\label{s2}

Throughout this paper, $\mathcal{H}$ denotes a real Hilbert space endowed with the inner product $\langle \cdot,\cdot\rangle$ and the associated norm $\|\cdot\|$.

Let $F:\mathcal{H}\rightrightarrows \mathcal{H}$ be a set-valued operator/multifunction. The {\it domain}, {\it range}, {\it graph}, and set of {\it zeros} of $F$ are defined, respectively, by
\begin{align*}
\dom F: &= \{x\in\mathcal{H}\;|\;F(x)\neq\varnothing\},\\
\ran F: &= \bigcup_{x\in\mathcal{H}}F(x),\\
\gra F: &= \{(x,y)\in\mathcal{H}\times\mathcal{H}\;|\;y\in F(x)\},\\
\zer F: &= \{x\in\mathcal{H}\;|\;0\in F(x)\}.
\end{align*}
The {\it inverse} operator of $F$ is defined by
\[
F^{-1}(y)=\{x\in\mathcal{H}\;|\;y\in F(x)\}.
\]
It is straightforward to verify that
\[
\dom F=\ran F^{-1},
\qquad
\ran F=\dom F^{-1}.
\]
The operator $F$ is said to be {\it monotone} if we have
\[
\langle x^*-y^*,x-y\rangle\geq 0\;\mbox{ for all }\;(x,x^*),(y,y^*)\in\gra F.
\]
The operator \(F\) is called {\it maximally monotone}
if \(F\) is monotone and  there is no monotone operator
$G :\mathcal{H}\rightrightarrows \mathcal{H} $ such that
\[
\operatorname{gra} F \subsetneq \operatorname{gra} G.
\]
Furthermore, $F$ is called {\it $\alpha$-strongly monotone} with modulus $\alpha>0$ if
\[
\langle x^*-y^*,x-y\rangle
\geq \alpha\|x-y\|^2
\]
whenever $(x,x^*),(y,y^*)\in\gra F$. The {\it resolvent} of $F$ is given by
\[
J_F:=(I+F)^{-1}.
\]

It is well known by the classical Minty theorem that if $F$ is maximally monotone, then $J_F$ is single-valued, everywhere defined on $\mathcal{H}$, and firmly nonexpansive. Furthermore, if $F$ is $\alpha$-strongly monotone, then the operator $J_F$ is $\frac{1}{1+\alpha}$-Lipschitz continuous and hence is contractive; see, e.g., 
\cite{Bauschke,Nesterov,Reich1,Rockafellar} for more details and discussions. 

Finally in this section, we recall several versions of the notions of {\it R-continuity} and the {\it compact  R-continuity} of a set-valued mapping $\mathcal{A}:\mathcal{H} \rightrightarrows \mathcal{H}$ at a point ${\bar x}$ in its domain. These notions introduced and developed in  \cite{L1,LMT,LT} provide, in particular, useful tools for the convergence analysis of the  algorithms. 

\begin{definition}\label{rdef}
The set-valued mapping \noindent $\mathcal{A}:\mathcal{H} \rightrightarrows \mathcal{H}$ is called {\sc R-continuous} at ${\bar x}\in\dom\,\mathcal{A}$ if there exist a number $\sigma>0$ and an {increasing} function $\rho: \mathbb{R}^+\to \mathbb{R}^+$ satisfying $\lim_{r\to 0^+}\rho(r)=\rho(0)=0$ such that we have the inclusion
\begin{equation}\label{rcon}
\mathcal{A}(x) \subset \mathcal{A}({\bar x}  )+\rho(\Vert x-{\bar x}   \Vert)\ball\;\mbox{ for all }\;x\in  \ball({\bar x} ,\sigma)
\end{equation}
with the {\sc continuity modulus function} $\rho$ and {\sc radius} $\sigma$. When $\sigma=\infty$, $\mathcal{A}$ is  called {\sc globally R-continuous} at ${\bar x}$. The inclusion in \eqref{rcon} means that for each $y\in \mathcal{A}(x)$, there exists ${\bar y}  \in  \mathcal{A}({\bar x} )$ satisfying the estimate 
$$
\Vert y-{\bar y}  \Vert\le \rho(\Vert x-{\bar x}   \Vert)
$$ 
for all $x\in \ball({\bar x},\sigma)$. We further say that 
\begin{itemize}
\item $\mathcal{A}$ is {\sc $R$-Lipschitz continuous} at ${\bar x}  $ with modulus {$L>0$}  if $\rho( r)=Lr$.
\item $\mathcal{A}$ is {\sc $R$-H\"older continuous} at ${\bar x}  $  if $\rho( r)=Lr^\theta$ for some $L>0, \theta>0$.
\end{itemize}
\end{definition}\vspace*{0.05in}

\begin{definition}\label{crdef}
A set-valued mapping \noindent $\mathcal{A}:\mathcal{H} \rightrightarrows \mathcal{H}$ is  called {\sc compactly 
R-continuous} at ${\bar x}\in\dom\,\mathcal{A}$ if for every compact set $K\subset {\mathcal{H}}$, there exists a number $\sigma>0$ such that we can find an {increasing}  function $\rho: \mathbb{R}^+\to \mathbb{R}^+$ satisfying $\lim_{r\to 0^+}\rho(r)=\rho(0)=0$ with
\begin{equation*}\label{rconc}
\mathcal{A}(x)\cap K \subset \mathcal{A}({\bar x}  )+\rho(\Vert x-{\bar x}   \Vert)\ball\;\mbox{ for all }\;x\in  \ball({\bar x} ,\sigma).
\end{equation*}
Similarly to Definition~{\rm\ref{rdef}}, we say that 
\begin{itemize}
\item $\mathcal{A}$ is {\sc compactly $R$-Lipschitz continuous} at ${\bar x} $ with modulus $L>0$ if $\rho( r)=Lr$.

\item $\mathcal{A}$ is {\sc compactly $R$-H\"older continuous} at ${\bar x}$ if $\rho( r)=Lr^\theta$ for some 
$L, \theta>0$.
\end{itemize}
\end{definition}\vspace*{0.05in}

\begin{definition}\label{closed graph} Given a set-valued mapping $A:\mathcal{H} \rightrightarrows \mathcal{H}$, we say that the {\sc graph} of ${\cal A}$ is $($sequentially$)$  {\sc closed at a  point} $\bar x\in\dom\,{\cal A}$ if for any sequences $x_k\to\bar x$ and any sequence $y_k\in{\cal A}(x_k)$ converging to some $y$ as $k\to\infty$, we have $y\in{\cal A}(\bar x)$.
\end{definition}

Finally in this section, we formulate several {results} related to R-continuity and compactly R-continuity obtained in the recent papers \cite{LMT,LT}.

\begin{proposition}\label{localc}{\rm\cite{LMT}}
If $\mathcal{A}$ is locally compact around  zero, i.e.,  there is $\sigma>0$ such that {$\mathcal{A}( \sigma \ball\setminus \{0\})$} belongs to a compact set, then the compact R-continuity of ${\cal A}$ at zero yields the R-continuity  {of ${\cal A}$} at  this point.
\end{proposition}

\begin{theorem}\label{compact}{\rm\cite{LT}}
If $\mathcal{A}:  \mathcal{H}\rightrightarrows \mathcal{H}$ has the closed graph  at zero and is  locally compact  around zero, then $\mathcal{A}$ is R-continuous at zero.
\end{theorem}

\begin{theorem}\label{compact1}{\rm\cite{LMT}}
The set-valued mapping $\mathcal{A}:\mathcal{H}\rightrightarrows \mathcal{H}$ is compactly R-continuous at zero {and $\mathcal{A}(0)$ is closed} if and only if its graph is closed at this point.
\end{theorem}

\section{Convergence of \textbf{IPPA}}\label{s3}

In this section, we study the convergence of \textbf{IPPA} applied to  the Tikhonov regularized problem for finding an approximate solution to (\ref{main}).

\begin{theorem}\label{thm-prox}
Suppose that $A:\mathcal{H}\rightrightarrows \mathcal{H}$ is a maximally monotone operator, and let
\[
S:=\zer A\neq\varnothing.
\]
Assume further that $A^{-1}$ is R-continuous at $0$ with the modulus function $\rho$. Let $(x_k)$ be the sequence generated by \textbf{IPPA} as follows:
\[
x_{k+1}
=
J_{\gamma \widetilde{A}}x_k+y_{k+1}
=
(I+\gamma \widetilde{A})^{-1}x_k+y_{k+1},
\qquad x_0\in\mathcal{H},
\]
with the notation
\[
\widetilde{A}=A+\epsilon I,
\]
and the bounded-error estimate
\[
\|y_{k+1}\|\leq \delta,
\]
where $\epsilon,\delta>0$ are some fixed parameters. Then it holds that
\begin{equation*}
\limsup_{k\to\infty} d(x_k,S)
\leq
\delta(1+\gamma^{-1}\epsilon^{-1})
+
\rho(\epsilon a)
\end{equation*}
with the notation
\[
a:=\|\widetilde{x}\|,
\] 
where $\widetilde{x}$ stands for the minimal norm element of $S$. Moreover, by choosing
\[
\delta=\epsilon^2,
\]
we get the convergence condition
\begin{equation*}
\limsup_{k\to\infty} d(x_k,S)
\leq
\epsilon^2+\gamma^{-1}\epsilon+\rho(\epsilon a)\to 0 \;\; {\rm as}\; \;\epsilon\to 0.
\end{equation*}
\end{theorem}

\begin{proof}
Consider the sequence
\[
w_k:=x_k-y_k
\]
along which we have the equalities
\begin{align*}\nonumber
w_{k+1}
&=
J_{\gamma \widetilde{A}}x_k
-
J_{\gamma \widetilde{A}}(x_k-y_k)
+
J_{\gamma \widetilde{A}}(x_k-y_k) \\
&=
J_{\gamma \widetilde{A}}x_k
-
J_{\gamma \widetilde{A}}(x_k-y_k)
+
J_{\gamma \widetilde{A}}(w_k).
\end{align*}
Denote by $x^*$ the unique solution to the Tikhonov-regularized problem
\[
0\in \widetilde{A}(x^*)
\quad\Longleftrightarrow\quad
x^*=J_{\gamma \widetilde{A}}x^*.
\]
Observe that $\widetilde{A}$ is maximally and $\epsilon$-strongly monotone, and hence the resolvent $J_{\gamma \widetilde{A}}$ is $\kappa$-Lipschitz continuous with the modulus
\[
\kappa:=\frac{1}{1+\gamma\epsilon}<1.
\]
Therefore, we arrive at the relationships
\begin{align*}\nonumber
\|w_{k+1}-x^*\|
&=
\Bigl\|
J_{\gamma \widetilde{A}}x_k
-
J_{\gamma \widetilde{A}}(x_k-y_k)
+
J_{\gamma \widetilde{A}}(w_k)
-
J_{\gamma \widetilde{A}}x^*
\Bigr\| \\\nonumber
&\leq
\|J_{\gamma \widetilde{A}}x_k
-
J_{\gamma \widetilde{A}}(x_k-y_k)\|
+
\|J_{\gamma \widetilde{A}}(w_k)
-
J_{\gamma \widetilde{A}}x^*\| \\\nonumber
&\leq
\kappa\|y_k\|
+
\kappa\|w_k-x^*\| \\
&\leq
\kappa\delta+\kappa\|w_k-x^*\|.
\end{align*}
Iterating the resulting inequality yields
\begin{align*}\nonumber
\|w_{k+1}-x^*\|
&\leq
\kappa^{k+1}\|w_0-x^*\|
+
\kappa\delta(1+\kappa+\cdots+\kappa^k) \\\nonumber
&\leq
\kappa^{k+1}\|x_0-x^*\|
+
\frac{\kappa\delta}{1-\kappa} \\
&\leq
\kappa^{k+1}\|x_0-x^*\|
+
\delta\gamma^{-1}\epsilon^{-1},
\end{align*}
which brings us to the estimate
\[
\limsup_{k\to\infty}\|w_k-x^*\|
\leq
\delta\gamma^{-1}\epsilon^{-1}.
\]
Observe furthermore that
\[
x^*\in A^{-1}(-\epsilon x^*).
\]
Since $A^{-1}$ is R-continuous at $0$, {this} tells us that
$$
x^*\in A^{-1}(-\epsilon x^*)\subset A^{-1}(0)+\rho(\epsilon \Vert x^* \Vert)\ball=S+\rho(\epsilon \Vert x^* \Vert)\ball,
$$
which implies in turn that
\[
d(x^*,S)
\leq
\rho(\epsilon\|x^*\|).
\]
Moreover, it follows from \cite{L1} that
\[
\|x^*\|
\leq
\|\widetilde{x}\|
=
a,
\]
and thus we arrive at the estimate
\[
d(x^*,S)\leq \rho(\epsilon a).
\]
We have the inequality
\begin{align*}
d(x_k,S)
\leq
\|x_k-w_k\|
+
\|w_k-x^*\|
+
d(x^*,S). 
\end{align*}
Passing there to the limit as $k\to\infty$ gives us the claimed condition
$$
\limsup_{k\to\infty} d(x_k,S)\leq
\delta(1+
\gamma^{-1}\epsilon^{-1})
+
 \rho(\epsilon a).
$$
Letting finally $\delta= \epsilon^2$ completes the proof of the theorem.
\end{proof}

\begin{remark} \label{rm1}
{\rm The following remarks are in order:

{\bf(i)} Theorem~\ref{thm-prox} shows that for all $k$ sufficiently large, the iterate $x_k$ is an {\it approximate solution} to the original inclusion  problem \eqref{main} provided that $\delta$ is sufficiently small and hence $\epsilon=\sqrt{\delta}$ is small as well.

{\bf(ii)} If $\mathcal{H}=\mathbb{R}^n$, it is sufficient to assume that $A^{-1}$ is {\it compactly R-continuous} at $0$ since $x^*$ belongs to the compact set $a\mathbb{B}$. Then this condition  automatically holds since $A$ is a maximally monotone operator, which implies that the graph of $A^{-1}$  is closed  and hence $A^{-1}$ is compactly R-continuous at $0$ by Theorem~\ref{compact1}. 

{\bf(iii)}  In practice, {\it errors cannot be summable}. Indeed, just simple functions such as $\sin x, e^x, \sqrt{x},\ldots$ have small positive errors. This indicates that the summable error assumption is not the reason to ensure the stability of {\bf PPA} but the {\it strong monotonicity} is, which makes the resolvent to be {\it contractive}. Note that even if the operator $A$ is only maximally monotone, there is some strong monotonicity part of $A$. {This} means that $\Vert J_{\gamma A}x-J_{\gamma A}y\Vert \le \kappa \Vert x-y\Vert$ for some $\kappa\in (0,1)$, $\gamma>0,$ and  $ x, y$ belong to the designated region.  Therefore, we also obtain the convergence to approximate solutions if {\bf PPA} is applied to the original operator $A$ with some suitable initial point $x_0$ and stepsize $\gamma$. For instance, if $A(x)=x-\sin x$, then it is maximally monotone but not strongly monotone on $[-\pi/2,\pi/2]$. However, this operator is strongly monotone on $C=\{x: \epsilon\le\vert x \vert\le  \pi/2\}$ for all small positive $\epsilon$. Thus if $x_0\in C$, then {\bf PPA} applied to $A$ converges to an appropriate approximation of the exact solution $x=0$ to \eqref{main} in the presence of nonsummable errors.

{\bf(iv) }  {In  \cite{zas10}, Zaslavski considered the {\bf IPPA} applied to a convex, lower semicontinuous function $f: H\to \R\cup \{\infty\}$ with $\inf_{x\in H} f(x)>-\infty$ under the rather restrictive assumption
\beq\label{coer}
\lim_{\Vert x \Vert\to \infty} f(x)=\infty.
\eeq
Notice that convex quadratic functions which are not strongly convex do not necessarily satisfy this assumption. For example, consider the function
$$
f:\mathbb{R}^3\to\mathbb{R},
\qquad
f(x)=x_1^2+x_2^2,
$$
where $x=(x_1,x_2,x_3)$. Then
$$
f(0,0,n)=0,
\qquad
\|(0,0,n)\|=n\to\infty
\quad\text{as }n\to\infty.
$$
Thus, (\ref{coer}) is not satisfied.
 Moreover, even when \eqref{coer} holds, the number of iterations required to obtain an approximate solution may be very large; see \cite[Theorem 1.1]{zas10} and \cite[Theorem 1.1]{zas12}.  We can prove that the convergence rate of {\bf PPA} to obtain approximate solutions is $O(\frac{1}{\sqrt{k}})$ (see \cite{Brezis,Gu}) and so is {\bf IPPA} by using \cite[Lemma 2.4]{zas12}. Indeed if 
 $$
 x_{k+1}=J_{\gamma {A}}x_k,
 $$
 it is classical that 
 $$
 \frac{ x_{k+1}- x_{k}}{\gamma}\in -A( x_{k+1}),
 $$
 and
 $$
 \Vert x_{k+1}-x^*\Vert^2\le  \Vert x_{k}-x^*\Vert^2- \Vert x_{k+1}-x_k\Vert^2,
 $$
 for some $x^*\in A^{-1}(0)$.
 Therefore
 $$
 \sum_{i=0}^{k-1}\Vert x_{i+1}-x_i\Vert^2\le \Vert x_0-x^*\Vert^2-\Vert x_{k}-x^*\Vert^2\le \Vert x_0-x^*\Vert^2:=C.
 $$
 Since $\Vert x_{i+1}-x_i\Vert=\Vert J_{\gamma A}x_{i}-J_{\gamma A}x_{i-1}\Vert\le \Vert x_{i} -x_{i-1}\Vert$, we deduce that 
 $$
\Vert x_{k}-x_{k-1}\Vert\le \frac{C}{\sqrt{k}},
 $$
 and the conclusion follows. 
}

}
\end{remark}

\section{Convergence of \textbf{ITA}}\label{s4}

In this section, we study the {\it composite inclusion problem} in the summation form 
\beq\label{sum}
0\in Ax+Bx,
\eeq
where $A:\mathcal{H}\rightrightarrows \mathcal{H}$ is a maximally monotone operator while $B:\mathcal{H}\to\mathcal{H}$ is an $L$-Lipschitz continuous and maximally monotone one. The {\it regularized inclusion} problem associated with \eqref{sum} is given as follows
\begin{equation}\label{tseng1}
0\in Ax+Bx+\epsilon x
=
\widetilde{A}x+Bx,
\end{equation}
where $\widetilde{A}=A+\epsilon I$ for some small $\epsilon>0$.

Consider the {\bf Inexact Tseng Algorithm (ITA)} to solve the regularized inclusion \eqref{tseng1} associated with \eqref{sum} and formulated as (see, e.g., \cite{Adeolu,tseng})
$$
\begin{cases}
y_{k}=J_{\gamma \tilde{A}}(x_k-\gamma Bx_k)+z_k,\\
x_{k+1}= y_k+\gamma (Bx_k-By_k),
\end{cases}
$$
where $z_k$ is the error in the step using the resolvent. Denoting $w_k:=y_k-z_k$ for $k\ge 0$ and recalling that $B$ is Lipschitz continuous, we can rewrite {\bf ITA} in the form
$$
\begin{cases}
w_{k}=J_{\gamma \tilde{A}}(x_k-\gamma Bx_k), \;k\ge 0,\\
x_{k+1}= w_k+\gamma (Bx_k-Bw_k)+t_{k+1},
\end{cases}
$$
where $\Vert t_{k+1}\Vert\le \delta$ for some numbers $\gamma, \delta>0$. The latter admits the equivalent representation
$$
x_{k+1}= Tx_k+t_{k+1},\;k\ge 0,\;x_0\in \mathcal{H},
$$
where the well-defined mapping $T: \mathcal{H}\to \mathcal{H}$ is given by
\beq\label{det}
\begin{cases}
y=J_{\gamma \tilde{A}}(x-\gamma Bx),\\
Tx= y+\gamma (Bx-By).
\end{cases}
\eeq

The following lemma of its own interest plays a crucial role in establishing the convergence of {\bf ITA} in the subsequent Theorem~\ref{thm:ita}.

\begin{lemma}\label{ab}
Let $x^*$ be the unique fixed point of $T$. Choose $\gamma>0$ such that $\gamma^2 L^2\le 1-2\gamma\varepsilon $. Then the operator $\Vert T(\cdot)-x^*\Vert: \mathcal{H}\to \R_+$ is $\kappa$-Lipschitz continuous with the modulus
$$
\kappa=\sqrt{1-\gamma\epsilon}<1.
$$
\end{lemma}
\begin{proof}
It is easy to see that $T$ is single-valued being defined on the whole space $\mathcal{H}$ since $\tilde{A}$ is a maximally monotone operator. Note that $\tilde{A}+B$ is strongly maximally monotone. Thus $T$ has a unique fixed point $x^*$, which satisfies the inclusion $0\in \tilde{A}x^*+Bx^*$. Let $x\in \mathcal{H}$ and $y=J_{\gamma \tilde{A}}(x-\gamma Bx)$. Then we clearly get
\beq\label{discr}
\frac{y-x}{\gamma}+Bx\in -\tilde{A}y
\eeq
together with the aforementioned inclusion
$$
Bx^*\in -\tilde{A}x^*.
$$
Since $\tilde{A}$ is $\epsilon$-strongly monotone, it follows that
\[
\left\langle
\frac{y-x}{\gamma}+{B}x-{B}x^*,
\, y-x^*
\right\rangle \le -\epsilon \Vert y-x^*\Vert^2.
\]
Combining the obtained inequality with the  monotonicity of $B$ tells us that
\begin{equation}\label{eq0}
\left\langle
\frac{y-x}{\gamma}+{B}x-{B}y,
\, y-x^*
\right\rangle \le -\epsilon \Vert y-x^*\Vert^2.
\end{equation}
By $\gamma>0$, we deduce from (\ref{det}) and \eqref{eq0} that
\[
\langle Tx-x,  y-x^*\rangle \le -\gamma \epsilon \Vert y-x^*\Vert^2,
\]
which implies in turn that 
\baqn
&&\langle Tx-x^*, y-x^*\rangle\le \langle x-x^*, y-x^*\rangle-\gamma \epsilon \Vert y-x^*\Vert^2\\
&\iff& \Vert Tx-x^*\Vert^2+\Vert y-x^*\Vert^2-\Vert Tx-y\Vert^2\\
&\le&\Vert x-x^*\Vert^2+\Vert y-x^*\Vert^2-\Vert x-y\Vert^2-2\gamma \epsilon \Vert y-x^*\Vert^2\\
&\iff& \Vert Tx-x^*\Vert^2\le \Vert x-x^*\Vert^2+\gamma^2 \Vert Bx-By\Vert^2-\Vert x-y\Vert^2-2\gamma \epsilon \Vert y-x^*\Vert^2.
\eaqn
Remembering that $B$ is $L$-Lipschitz continuous and that 
$$
\gamma^2 L^2\le 1-2\gamma\varepsilon
$$
brings us to the inequalities
$$
\Vert Tx-x^*\Vert^2\le \Vert x-x^*\Vert^2-2\gamma \epsilon (\Vert x-y\Vert^2+\Vert y-x^*\Vert^2)\le(1-\gamma \epsilon) \Vert x-x^*\Vert^2,
$$
which verify the claim of the lemma.
\end{proof}

Now we are ready to establish the convergence of {\bf ITA} with a constructive estimate for approximate solutions to \eqref{sum}.

\begin{theorem}\label{thm:ita} Let $A:\mathcal{H}\rightrightarrows \mathcal{H}$ be an arbitrary maximally monotone operator, and let the operator $B:\mathcal{H}\to\mathcal{H}$ be an $L$-Lipschitz continuous and maximally monotone. Denote 
\[
S:=\zer (A+B)\neq\varnothing
\]
and assume that $(A+B)^{-1}$ is R-continuous at $0$ with modulus function $\rho$. Let $(x_k)$ be the sequence generated by the {\bf ITA} applied to the regularized problem \eqref{tseng1}. Choose $\gamma>0$ such that $\gamma^2 L^2\le 1-2\gamma\varepsilon $.
Then we have the estimate
\begin{equation}
d(x_k,S)
\leq
\delta(1+2\gamma^{-1}\epsilon^{-1})
+
\rho(\epsilon a)
\end{equation}
with the notation
\[
a:=\|\widetilde{x}\|,
\]
where $\widetilde{x}$ stands for the minimal norm element of $S$. 
By choosing
\[
\delta=\epsilon^2,
\]
we obtain the convergence condition
\begin{equation*}
d(x_k,S)
\leq
\epsilon^2+2\gamma^{-1}\epsilon+\rho(\epsilon a)\to 0 \;\; {\rm as}\; \;\epsilon\to 0.
\end{equation*}
\end{theorem}
\begin{proof}
It follows that
$$
x_{k+1}= Tx_k+t_{k+1},\;k=0,1,\ldots.
$$
Let $z_k:=x_{k}-t_k$, and let $x^*$ be the unique fixed point of $T$ that also satisfies the inclusion $0\in \tilde{A}x^*+Bx^*$. Using Lemma~\ref{ab}, we get
\begin{align*}
\|z_{k+1}-x^*\|
=\| Tx_k-x^*\|\le \kappa\|x_k-x^*\|
\leq
\kappa(\|z_k-x^*\|+\|t_k\|)\le \kappa(\|z_k-x^*\|+\delta),
\end{align*}
where $\kappa=\sqrt{1-\gamma\epsilon}<1$. The induction procedure gives us the estimates
\begin{align*}
\|z_{k+1}-x^*\|
&\leq
\kappa^{k+1}\|z_0-x^*\|
+
\kappa\delta(1+\kappa+\cdots+\kappa^k) \\
&\leq
\kappa^{k+1}\|x_0-x^*\|
+
\delta\frac{\kappa}{1-\kappa}
\end{align*}
which consequently yield the inequalities
\[
\limsup_{k\to\infty}\|z_k-x^*\|
\leq
\delta\frac{\kappa}{1-\kappa}\le\delta\frac{1+\kappa}{1-\kappa^2} \le 2\delta \gamma^{-1}\epsilon^{-1}.
\]
It is {obvious} to observe that
\[
x^*\in (A+B)^{-1}(-\epsilon x^*).
\]
Since $(A+B)^{-1}$ is R-continuous at $0$, it follows that
\[
d(x^*,S)
\leq
\rho(\epsilon\|x^*\|)\le\rho(\epsilon a).
\]
Note furthermore that
\begin{align*}
d(x_k,S)
\leq
\|x_k-z_k\|
+
\|z_k-x^*\|
+
d(x^*,S),
\end{align*}
which implies in turn that
$$
\limsup_{k\to\infty} d(x_k,S)\leq
\delta(1+
2\gamma^{-1}\epsilon^{-1})
+
\rho(\epsilon a).
$$
By choosing $\delta=\epsilon^{2}$, we complete the proof of the theorem.
\end{proof}

\section{Concluding Remarks and Future Research}\label{s5}

In this paper, we propose a new approach to the design and justification of {\it practical inexact algorithms} for solving optimization-related problems governed by monotone inclusions in Hilbert spaces. The novel characteristic feature of such algorithms  is incorporating {\it nonsummable errors}, which is the realistic case in practical applications. Our approach is based on the {\it Tikhonov regularization} of monotone inclusions and the usage of quite recent results on {\it R-continuity} of set-valued mappings.

The main attention of this paper is paid to inexact versions of the two algorithms named the {\it Inexact Proximal Point Algorithm}  and the {\it Inexact Tseng Algorithm} in the presence of nonsummable errors, which are more realistic in comparison with the standard summability assumption on errors in computations.

Our analysis shows that the regularization term plays a crucial role in ensuring the stability of the algorithms under bounded perturbations. The proposed technique is sufficiently flexible and can be potentially extended to the study of other important inexact algorithms to solve optimization-related problems in the presence of nonsummable perturbations under bounded-error conditions. This goal constitutes a promising direction for { {\it future research}}.

\end{document}